\documentclass[12pt,a4paper]{amsart}
\usepackage[T1]{fontenc}
\usepackage[utf8]{inputenc}
\usepackage{amsmath}
\usepackage{amsfonts}
\usepackage{amssymb}

\usepackage{mathabx}
\usepackage{mathrsfs}

\usepackage{amsthm}

\usepackage{enumerate}

\usepackage{todo}

\usepackage[twoside,width=16.00cm, height=24.00cm]{geometry} 
\usepackage{layout}

\usepackage{mathtools}

\usepackage{tikz-cd}
\tikzcdset{row sep/normal=3.6em, column 
	sep/normal=3.6em}

\newtheorem{theorem}{Theorem}[section]
\newtheorem*{thm}{Theorem}
\newtheorem{proposition}[theorem]{Proposition}
\newtheorem{lemma}[theorem]{Lemma}

\theoremstyle{definition}
\newtheorem{definition}[theorem]{Definition}
\newtheorem{example}[theorem]{Example}

\newtheorem{remark}[theorem]{Remark} 
\newtheorem{notation}[theorem]{Notation}

\newtheorem{corollary}[theorem]{Corollary}
\newtheorem{construction}[theorem]{Construction}

\usepackage{newmacros}

\begin{document}

\title{A Categorical View of Varieties of Ordered Algebras}
\author{J.~Adámek}
\address{Department of Mathematics, Faculty of Electrical Engineering, Czech Technical University in Prague, Czech Republic, and Institute for Theoretical Computer Science, Technical University Braunschweig, Germany}
\email{j.adamek@tu-bs.de}
\author{M.~Dostál}
\author{J.~Velebil}
\address{Department of Mathematics, Faculty of Electrical Engineering, Czech Technical University
	in Prague, Czech Republic}
\email{$\{$dostamat,velebil$\}$@fel.cvut.cz}
\thanks{J.~Ad\'{a}mek and M.~Dost\'{a}l acknowledge the support
	of the grant No.~19-0092S
	of the Czech Grant Agency.} 

\begin{abstract}
It is well known that classical varieties of $\Sigma$-algebras correspond bijectively to finitary monads on $\Set$.
We present an analogous result for varieties of ordered $\Sigma$-algebras, i.e., categories presented by inequations between $\Sigma$-terms.
We prove that they correspond bijectively to strongly finitary monads on $\Pos$.
That is, those finitary monads which preserve reflexive coinserters.
We deduce that strongly finitary monads have a coinserter presentation, analogous to the coequaliser presentation of finitary monads due to Kelly and Power.
We also show that these monads are liftings of finitary monads on $\Set$.
\end{abstract}

\maketitle

\begin{center}
\emph{Dedicated to John Power,\\ from whom we have learned so much,\\ on the occasion of his 60th birthday}
\end{center}

\section{Introduction}

Varieties of ordered algebras, i.e., classes of ordered $\Sigma$-algebras (for a finitary signature $\Sigma$) presented by inequations between $\Sigma$-terms, play an important role in universal algebra and computer science.
Example: ordered monoids with bottom as the unit, $e$, are presented by the inequation $e \leq x$ and the usual equations for classical monoids.
For every variety $\Vvar$ free algebras exist on all posets, that is, the forgetful functor $\Vvar \to \Pos$ has a left adjoint.
The corresponding monad $\Tmon$ on $\Pos$ will be proved to be \emph{strongly finitary}, which means that its underlying endofunctor $T$ preserves
\begin{enumerate}
	\item filtered colimits, and
	\item coinserters of reflexive pairs.
\end{enumerate}
In the above example of ordered monoids $\Tmon$ is a lifting of the word monad (of monoids) on $\Set$.
For every poset $X$ we have the poset $TX = X^*$ with the following order: a word $x_0 \dots x_{n-1}$ is smaller or equal to a word $w$ iff $w$ decomposes as $w = w_0 \dots w_{n-1}$ and each $w_i$ contains a letter $y_i \in X$ with $x_i \leq y_i$ in $X$.

Conversely, given a strongly finitary monad $\Tmon$ on $\Pos$, its Eilenberg-Moore category $\Pos^\Tmon$ will be proved to be isomorphic to a variety of ordered algebras.
This leads to the following main result of our paper:
\begin{thm}
The category of varieties of ordered algebras (with concrete functors as morphisms) is dually equivalent to the category of strongly finitary monads on $\Pos$.
\end{thm}

We thus obtain a bijective correspondence between varieties of ordered algebras and strongly finitary monads on $\Pos$.
This is analogous to the well-known correspondence between (classical) varieties and finitary monads on $\Set$, up to natural isomorphism.

Moreover, every variety of ordered algebras is a lifting of a classical variety.
This follows from the above bijective correspondence and the fact we prove that every strongly finitary monad $\Tmon$ on $\Pos$ is a lifting of a finitary monad $\wt{\Tmon}$ on $\Set$: for every poset $X$ the underlying set of $TX$ is $\wt{T}|X|$, and the underlying maps of $\eta_X$ and $\mu_X$ are $\wt{\eta}_{|X|}$ and $\wt{\mu}_{|X|}$, respectively.
Naturally, one classical variety can have many liftings, consider e.g.\ ordered monods (a 'minimal' lifting of the variety of monoids), compared with our example above.

\subsection*{Related results}
The bijective correspondence between varieties of ordered algebras and strongly finitary monads has been established already by Kurz and Velebil~\cite{kurz+velebil}.
However, the proof there was derived from technically involved results concerning the exactness (in $\Pos$-enriched sense) of these varieties.
Our present proof is much simpler.

Strongly finitary monads on enriched categories were studied by Kelly and Lack~\cite{kelly+lack:strongly-finitary}.
When specialised to $\Pos$ (enriched over itself as a cartesian closed category), their results yield a bijection between strongly finitary monads and equationally (!) presented classes of $\Sigma$-algebras.
However, here $\Sigma$ means a much more complex concept of signature, following the paper of Kelly and Power~\cite{kelly+power:adjunctions}: let $\Pos_\fp$ be a set of finite posets representing all of them up to isomorphism.
The signatures in $\Pos$ introduced in~\cite{kelly+power:adjunctions} are collections $\Sigma = (\Sigma_n)_{n \in \Pos_\fp}$ of posets $\Sigma_n$.
In the recent paper~\cite{adamek+chase+milius+schroder} finitary (ordinary as well as enriched) monads on $\Pos$ are studied.
They are related to inequationally specified classes of $\Sigma$-algebras for signatures $\Sigma$ that present a compromise between the classical signatures (used in the present paper) and those of Kelly and Power: they are collections of sets $\Sigma_n$ indexed by $n \in \Pos_\fp$.

\section{Finitary and Strongly Finitary Functors}

In the present section we recall finitary and strongly finitary endofunctors of $\Pos$.
We observe that a finitary endofunctor is strongly finitary iff it preserves reflexive coinserters.

\begin{remark}
\label{rem:intro-remarks}
\phantom{formatting}
\begin{enumerate}
	\item Throughout the paper we view $\Pos$ as the cartesian closed category with the hom-sets $\Pos(X,Y)$ ordered pointwise.
	All categories are understood to be enriched over $\Pos$.
	That is, hom-sets carry partial orders such that composition is monotone.
	All functors, limits, colimits and adjunctions are understood as enriched over $\Pos$.
	
	Thus when we say `endofunctor $H$ of $\Pos$' we automatically mean that it is locally monotone.
	Its underlying ordinary functor is denoted by $H_\ordi$.

	\item Colimits are understood to be weighted.
	Let us recall that for a given scheme, i.e., a small category $\D$, a \emph{weight} is a functor $\phi: \D^\op \to \Pos$.
	Example: given a poset $X$ and a diagram $D: \D \to \Pos$, the functor $\Pos(D\blank,X): \D^\op \to \Pos$ is a weight.
	The category of all weights is simply the functor category $[\D^\op,\Pos]$.
	
	A \emph{weighted colimit} of a diagram $D: \D \to \Pos$ of weight $\phi$ is a poset $\Colim{\phi}{D}$ together with an isomorphism
	\[
	\Pos(\Colim{\phi}{D},X) \cong [\D^\op,\Pos](\phi,\Pos(D\blank,X))
	\]
	natural in $X \in \Pos$.
	\item Every set is considered as a poset with the discrete order.
	In particular, every natural number $n$ is the discrete poset on the set $\{ 0, 1, \dots, n-1 \}$.
\end{enumerate}
\end{remark}

\begin{example}
\emph{Coinserters} are colimits of the scheme $\D$ given by a parallel pair
\[
\begin{tikzcd}
X
\arrow[r, "1", bend left]
\arrow[r, "0", swap, bend right]
&
Y
\end{tikzcd}
\]
The weight $\phi: \D^\op \to \Pos$ is as follows:
\[
	\tikz[
overlay]{
	\filldraw[fill=white,draw=black] (1.2,-1.2) rectangle (1.8,1.2);
}
\begin{tikzcd}[row sep=small, column sep=tiny]
	&
	\bullet
	\arrow[dd, no head]
	&& \\
	&
	&& \bullet \; \phi Y
	\arrow[ull, "\phi 1", swap]
	\arrow[dll, "\phi 0"]
	\\
	\phi X & \bullet &&
\end{tikzcd}
\]
Thus, a diagram in $\Pos$ is a parallel pair
\[
f_0,f_1: A \to B
\]
of monotone maps (considered as an ordered pair $(f_0,f_1)$, of course).
And the coinserter is a morphism $c: B \to C$ universal w.r.t.\ $c \comp f_0 \leq c \comp f_1$.
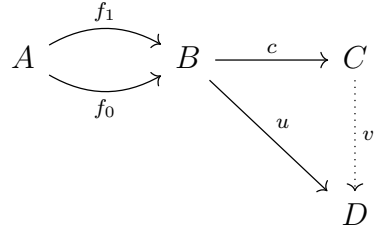
\[
\begin{tikzcd}
A
\arrow[r, "f_1", bend left]
\arrow[r, "f_0", swap, bend right]
& B
\arrow[r, "c"]
\arrow[rd, "u"]
& C
\arrow[d, dotted, "v"]
\\
& & D
\end{tikzcd}
\]
That is:
\begin{enumerate}
	\item for every morphism $u: B \to D$ with $u \comp f_0 \leq u \comp f_1$ there exists a unique morphism $v: C \to D$ with $u = v \comp c$, and
	\item  the map $u \mapsto v$ is monotone: given $u' = v' \comp c$, then $u \leq u'$ implies $v \leq v'$.
\end{enumerate}
\end{example}

\begin{remark}
	\label{rem:poset-canonical-presentation}
	Every finite poset $P$ is a \emph{canonical coinserter} of a parallel pair
	
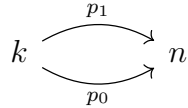
\[
	\begin{tikzcd}
		k
		\arrow[r, "p_1", bend left]
		\arrow[r, "p_0", swap, bend right]
		&
		n
	\end{tikzcd}
	\]
	of morphisms in $\N$.
	Let $n$ be the number of elements of $P$ and $k$ the number of comparable pairs in $P$.
	Thus we can assume that $P$ has elements $0, \dots, n-1$, and we can index all comparable pairs as follows
	\[
	{p_0(t)} \leq {p_1(t)} \text{ for } t = 0, \dots, k-1.
	\] 
	This defines functions $p_0, p_1: k \to n$. The coinserter of this pair is carried by the identity map:
	\[
	\begin{tikzcd}
		k
		\arrow[r, "p_1", bend left]
		\arrow[r, "p_0", swap, bend right]
		&
		n
		\arrow[r, "\id"]
		&
		P
	\end{tikzcd}
	\]
\end{remark}

\begin{notation}
Denote by
\[
J: \Pos_\fp \to \Pos
\]
the full embedding of a subcategory $\Pos_\fp$ representing all finite posets up to isomorphism.
\end{notation}

\begin{remark}
\phantom{formatting}
\begin{enumerate}
	\item $\Pos$ is a free completion of $\Pos_\fp$ under filtered conical colimits.
	In the realm of ordinary categories this follows from~\cite{adamek+rosicky} (Theorem~1.46) since $\Pos$ is a locally finitely presentable category with finite posets precisely the finitely presentable objects.
	Thus, given an ordinary category $\K$ with filtered colimits, for every ordinary functor $H: \Pos_\fp \to \K$ there exists an extension $H': \Pos \to \K$ preserving filtered colimits, unique up to natural isomorphism.
	Filtered conical colimits in $\Pos$ have the property that given a colimit cocone $c_i: C_i \to C$ ($i \in I$), then for two morphisms $u,v: C \to X$ we have $u \leq v$ iff $u \comp c_i \leq v \comp c_i$ for all $i \in I$.
	It follows that $H'$ is locally monotone whenever $H$ is.
	Thus, the statement above holds also in the enriched sense.

	\item Following Kelly~\cite{kelly:structures} we call an endofunctor of $\Pos$ \emph{finitary} iff its underlying ordinary endofunctor is finitary (i.e., preserves ordinary filtered colimits).
\end{enumerate}
\end{remark}

We now turn to strongly finitary functors.

\begin{notation}
\label{not:variables}
The full subcategory of $\Pos$ on natural numbers (Remark~\ref{rem:intro-remarks} (3)) is denoted by $\N$, and the full embedding by
\[
I: \N \to \Pos.
\]
\end{notation}

\begin{definition}[\cite{kelly+lack:strongly-finitary}]
An endofunctor $H$ of $\Pos$ is called \emph{strongly finitary} if it is the left Kan extension of its restriction to $\N$. More precisely:
\[
H = \Lan{I}{H \comp I}.
\]
\end{definition}

\begin{remark}
In ordinary categories \emph{sifted colimits} are colimits of diagrams whose schemes $\D$ are (small) sifted categories.
This means categories such that colimits of diagrams $D: \D \to \Set$ commute with finite products.

In our enriched setting, sifted colimits are introduced analogously.
A weight $\phi: \D^\op \to \Pos$ is called \emph{sifted} if the functor $\Colim{\phi}{\blank}: [\D,\Pos] \to \Pos$ preserves finite (conical) products.
Sifted colimits then are colimits weighted by sifted weights.
\end{remark}

\begin{example}
\phantom{formatting}
\begin{enumerate}
	\item Filtered colimits are clearly sifted (the corresponding weighted colimits in $\Pos$ commute with finite limits).
	\item A pair $f_0,f_1: A \to B$ is called reflexive if there exists $i: B \to A$ with $f_0 \comp i = \id_B = f_1 \comp i$.
	Coinserters of reflexive pairs are sifted colimits.
	The proof is completely analogous to the fact that in ordinary categories coequalisers of reflexive pairs are sifted colimits (\cite{adamek+rosicky+vitale:what-are-sifted}, Example~1.2).
	We speak about \emph{reflexive coinserters}.
	Example: the canonical coinserters (Remark~\ref{rem:poset-canonical-presentation}) are clearly reflexive.
\end{enumerate}
\end{example}

\begin{theorem}[\cite{bourke:thesis}, Corollary~8.45]
The following conditions are equivalent for endofunctors $H$ of $\Pos$:
\begin{enumerate}
	\item $H$ is strongly finitary,
	\item $H$ preserves sifted colimits,
	\item $H$ is finitary and preserves reflexive coinserters, and
	\item $H = \Lan{I}{H \comp I}$.
\end{enumerate}
\end{theorem}

\begin{proof}
	Every poset is a filtered colimit of its finite subposets, each of which is a coinserter as in Remark~\ref{rem:poset-canonical-presentation}. 
	
	Consequently, starting with the subcategory $\N$ we obtain all of $\Pos$ by reflexive coinserters and filtered colimits.
	In the terminology of~\cite{kelly:book} (Theorem~5.29), this states that the embedding $I: \N\to \Pos$ has a codensity presentation formed by filtered colimits and reflexive coinserters.
	By that theorem properties (1)-(4) are equivalent.
	
\end{proof}

\begin{remark}
	In (3) we can substitute reflexive coinserters by canonical coinserters, as is clear from the above proof.
\end{remark}

\begin{remark}
The above theorem is completely analogous to the fact proved in~\cite{adamek+rosicky+vitale:what-are-sifted} for ordinary endofunctors of categories with finite coproducts: preservation of sifted colimits is equivalent to the preservation of filtered colimits and reflexive coequalisers.
\end{remark}

\begin{example}
\label{ex:strongly-finitary-functors}
\phantom{formatting}
\begin{enumerate}
	\item The endofunctor $X \mapsto X^m$ (for $m \in \Nat$) of $\Pos$ is strongly finitary: it clearly preserves filtered colimits, and we verify that it also preserves the canonical coinserters of Remark~\ref{rem:poset-canonical-presentation}.
	Suppose $m = 2$.
	Then a comparable pair in $P \times P$ is a pair $(a,b)$ where the left-hand components of $a$ and $b$ are comparable in $P$, and thus have the form $x_{p_0(i)} \leq x_{p_1(i)}$ for some $i \leq k-1$.
	And the right-hand components have the form $x_{p_0(j)} \leq x_{p_1(j)}$ for some $j \leq k-1$.
	Thus the only comparable pairs of $P \times P$ are $(x_{p_0(i)},x_{p_0(j)}),(x_{p_1(i)},x_{p_1(j)})$.
	We conclude that the canonical coinserter of the poset $P \times P$ is given by $p_0 \times p_0, p_1 \times p_1: k \times k \to n \times n$.
	Analogously for $m > 2$.
	\item Coproducts of strongly finitary endofunctors are strongly finitary.
	Example:
	given a signature $\Sigma$, the corresponding polynomial functor $X \mapsto \coprod_{m \in \Nat} \Sigma_m \times X^m$ is strongly finitary.
	\item (Weighted) colimits of strongly finitary endofunctors are strongly finitary.
	\item A composite of strongly finitary endofunctors is strongly finitary.
\end{enumerate}
\end{example}

\begin{remark}
\label{rem:sf-endofunctor-sf-monad}
Every strongly finitary endofunctor $H$ of $\Pos$ generates a free monad whose underlying functor $\wh{H}$ is also strongly finitary.
Indeed, following~\cite{trnkova+adamek+koubek+reiterman}, $\wh{H}$ is a colimit in $[\Pos,\Pos]$ of the following $\omega$-chain
\[
\begin{tikzcd}
	\Id
	\arrow[r, "w_0"]
	& H + \Id
	\arrow[r, "w_1"]
	& H(H + \Id) + \Id
	\arrow[r, "w_2"]
	& \dots
\end{tikzcd}
\]
That is, the chain $W: \omega \to [\Pos,\Pos]$ has objects
\[
W_0 = \Id \text{ and } W_{n+1} = HW_n + \Id
\]
and morphisms
\[
w_0: \Id \to H + \Id \text{ the coproduct injection}
\]
and
\[
w_{n+1} = Hw_n + \id.
\]
Thus if $H$ is strongly finitary, so is each $W_n$ (by the preceding example).
Consequently, $\wh{H} = \colim W_n$ is strongly finitary.
\end{remark}

\begin{notation}
A monad whose endofunctor is strongly finitary is called a \emph{strongly finitary monad}.
We denote by
\[
\Mndsf(\Pos)
\]
the category of strongly finitary monads and monad morphisms.
\end{notation}

\begin{example}
\label{ex:free-monad-from-endofunctor}
The endofunctor $H_\Sigma$ generates the following free monad $\Tmon_\Sigma$ on $\Pos$: to every poset $X$ (of variables) it assigns the poset $T_\Sigma X$ of $\Sigma$-terms with variables from $X$.
That is, the underlying set is the smallest set containing $X$ and such that for every $\sigma \in \Sigma_n$ and every $n$-tuple $t_i$ in $T_\Sigma X$ we have $\sigma(t_i)$ in $T_\Sigma X$. 
This yields a structure of a $\Sigma$-algebra on $T_\Sigma X$.
The ordering of $T_\Sigma X$ is the smallest one such that $T_\Sigma X$ contains $X$ as a subposet, and all operations are monotone.
It follows from~\ref{rem:sf-endofunctor-sf-monad} that $\Tmon_\Sigma$ is strongly finitary.
\end{example}

\begin{remark}
It follows from Example~\ref{ex:strongly-finitary-functors} and Remark~\ref{rem:sf-endofunctor-sf-monad} that $\Mndsf(\Pos)$ has (weighted) colimits.
Indeed, given a diagram $D$ and a weight, the underlying diagram $D_\ordi$ in $\Endsf(\Pos)$ has a colimit $H$ which is strongly finitary. The free monad $\wh{H}$ is then a colimit of $D$ in $\Mndsf(\Pos)$, and it is strongly finitary.
\end{remark}

\section{From Varieties of Ordered Algebras to Strongly Finitary Monads}
\label{sec:varieties-to-sf-monads}

\begin{notation}
\label{not:sigma-algebra}
Let $\Sigma$ be a signature, i.e., a collection of sets $\Sigma_n$ (of $n$-ary operation symbols) indexed by $n \in \Nat$.
An ordered $\Sigma$-algebra is a poset $A$ together with a monotone map $\sigma_A: A^n \to A$ for every $n \in \Nat$ and $\sigma \in \Sigma_n$.
The category of ordered $\Sigma$-algebras and homomorphisms (i.e., monotone functions preserving the given operations) is denoted by $\Alg(\Sigma)$.
\end{notation}

\begin{remark}
\phantom{formatting}
\begin{enumerate}
	
	\item With $\Sigma$ we associate the \emph{polynomial functor} $H_\Sigma: \Pos \to \Pos$ given on objects by
	\[
	H_\Sigma X = \coprod_{n \in \Nat} \Sigma_n \times X^n
	\]
	and analogously on morphisms.
	By Example~\ref{ex:strongly-finitary-functors} (2), $H_\Sigma$ is strongly finitary.
	
	\item $\Alg(\Sigma)$ is clearly equivalent to the category of algebras for $H_\Sigma$, i.e., pairs $(A,\alpha)$ where $A$ is a poset and $\alpha: H_\Sigma A\to A$ is a monotone function. (Morphisms are monotone maps making the obvious square commutative.)
	
	\item It follows from~\cite{barr} that the category of algebras for an ordinary endofunctor $H$ is equivalent to the category of Eilenberg-Moore algebras for the free monad $\wh{H}$ (see Remark~\ref{rem:sf-endofunctor-sf-monad}).
	The same result holds for enriched endofunctors.
	In particular, we conclude
	\[
	\Alg(\Sigma) \simeq \Pos^{\Tmon_\Sigma}.
	\]
	
	\item The algebra $T_\Sigma X$ of terms (Example~\ref{ex:free-monad-from-endofunctor}) is a free $\Sigma$-algebra on $\eta_X: X \to T_\Sigma X$, the inclusion of variables:
	for every $\Sigma$-algebra $A$ and every monotone function $f: X \to A$ the unique extension to a homomorphism $f^\sharp: T_\Sigma X \to A$ is given by
	\[
	f^\sharp(\sigma(t_i)) = \sigma_A(f^\sharp(t_i)).
	\]	
\end{enumerate}
\end{remark}

\begin{definition}
\label{def:ordered-variety}
Let $V$ be a countably infinite set (of variables), $V = \{ x_n \mid n \in \Nat \}$.
An ordered pair of terms in $T_\Sigma V$ is called an \emph{inequation} and is written as $u \leq v$.
A $\Sigma$-algebra $A$ \emph{satisfies} $u \leq v$ iff every map $f: V \to |A|$ (interpretation of variables) fulfills $f^\sharp(u) \leq f^\sharp(v)$.

By a \emph{variety of ordered $\Sigma$-algebras} we understand a full subcategory of $\Alg(\Sigma)$ specified by a set of inequations.
\end{definition}

\begin{example}
	\phantom{formatting}
	\begin{enumerate}
		\item Ordered monoids are specified by the usual signature $\Sigma = \{ \cdot, e \}$ and the usual equations for monoids.
		The corresponding algebras are monoids with a partial order making the multiplication monotone (in both variables).
		
		This leads to the monad $\Tmon$ on $\Pos$ lifting the word monad on $\Set$ as follows:
		\[
		TX = X^*,
		\]
		the poset of words on $|X|$ ordered pointwise:
		\[
		x_0 x_1 \dots x_{n-1} \leq y_0 y_1 \dots y_{m-1} \; \text{ iff } \; n = m \text{ and } x_i \leq y_i \; \; (i < n).
		\]
		
		\item If we add to the equations above the inequation
		\[
		x \leq x \cdot y
		\]
		we obtain the variety of ordered monoids with $e$ the smallest element.
		That is, the above inequation is equivalent to
		\[
		e \leq y.
		\]
		Indeed the first inequation yields the latter one by putting $x = e$.
		Conversely, from $e \leq y$ we get $x = x \cdot e \leq x \cdot y$.
		
		The corresponding monad is the lifting of the word monad
		\[
		TX = X^*
		\]
		ordered as follows:
		\[
		x_0 x_1 \dots x_{n-1} \leq w \; \text{ iff } \; w=  w_0 w_1 \dots w_{n-1} \text{ and } w_i \text{ contains } y_i \text{ with } x_i \leq y_i \; \; (i < n).
		\]
		
		\item Bounded posets (with a least element $0$ and a largest element $1$) form a variety with $\Sigma$ given by nullary operations $0$,$1$ and the variety is presented by the inequations
		\[
		0 \leq x \text{ and } x \leq 1.
		\]
		This is a lifting of the variety of non-ordered algebras with two nullary operations.
	\end{enumerate}
\end{example}

\begin{remark}
\label{rem:varieties-sur-reflective}
Every variety of $\Sigma$-algebras is a reflective subcategory of $\Alg(\Sigma)$ with surjective reflections.

Indeed, since $H_\Sigma$ is a finitary endofunctor on a locally finitely presentable category, $\Alg(\Sigma) \cong H_\Sigma\text{-}\Alg$ is also locally finitely presentable, see~\cite{adamek+rosicky}, Remark 2.78.
In particular, it is complete and cowellpowered.
The factorisation system (epi, embedding) on $\Pos$ lifts, since $H_\Sigma$ preserves epimorphisms, to $\Alg(\Sigma)$.
Since a variety $\V$ is easily seen to be closed under products and subalgebras carried by embeddings, the surjective reflections follow, see~\cite{adamek+herrlich+strecker}, Theorem~16.8.
\end{remark}

\begin{construction}[see~\cite{bloom}]
\label{con:varietal-free-algebra}
For every variety $\Vvar$ of ordered algebras the free algebra $T_\Vvar X$ of $\Vvar$ on a poset $X$ can be constructed as follows.
	
Let $\Eeq_X$ be the collection of all inequations $s \leq t$ satisfied by all algebras of $\Vvar$, where $s,t \in T_\Sigma X$ are terms in variables from $X$.
Then $\Eeq_X$ is a preorder, i.e., a reflexive and transitive relation on $T_\Sigma X$.
Moreover, it is \emph{admissible} in the sense of Bloom~\cite{bloom}: given an $n$-ary symbol $\sigma \in \Sigma$ and $n$ pairs $s_i \leq t_i$ ($i<n$) in $\Eeq_X$, it follows that the pair $\sigma(s_i) \leq \sigma(t_i)$ also lies in $\Eeq_X$.
Indeed, given an algebra $A \in \Vvar$ and an interpretation $f: X \to |A|$, we know that the homomorphism $f^\sharp: T_\Sigma X \to A$ fulfils $f^\sharp (s_i) \leq f^\sharp (t_i)$ for all $i$, thus
\[
f^\sharp (\sigma(s_i)) = \sigma_{T_\Vvar X}(f^\sharp(s_i)) \leq \sigma_{T_\Vvar X}(f^\sharp(t_i)) = f^\sharp(\sigma(t_i)).
\]
Consequently, for the induced equivalence relation
\[
\Eeq_X^o = \Eeq_X \cap \Eeq_X^{-1}
\]
we obtain a $\Sigma$-algebra $T_\Vvar X$ on the quotient set
\[
|T_\Vvar X| = |T_\Sigma X|/\Eeq_X^o
\]
(of all equivalence classes $[t]$ of terms $t \in T_\Sigma X$).
The operations are as expected:
\[
\sigma_{T_\Vvar X} ([t_0], \dots, [t_{n-1}]) = [\sigma(t_0,\dots,t_{n-1})]
\]
for every $n$-ary $\sigma$ and all $n$-tuples $t_0,\dots,t_{n-1} \in T_\Sigma X$.
Finally, we consider $T_\Vvar X$ as a poset via
\[
[s] \leq [t] \text{ iff } (s,t) \in \Eeq_X.
\]
\end{construction}

The following theorem was stated by Bloom (\cite{bloom}, Theorem~2.2).
We present a full proof since we need it later, and the original proof was only a sketch.

\begin{theorem}
\label{thm:free-ordered-algebra}
The above ordered algebra $T_\Vvar X$ is a free algebra of the variety $\Vvar$ on the poset $X$ w.r.t.\ $\eta_X: X \to T_\Vvar X$ given by $x \mapsto [x]$.
\end{theorem}

\begin{proof}
	\phantom{formatting}
	\begin{enumerate}
		\item $T_\Vvar X$ is a well-defined ordered $\Sigma$-algebra.
		This follows easily from the fact that $\Eeq_X$ is an admissible preorder.
		
		\item $\Vvar$ has a free algebra on $X$ which is given by an admissible preorder $\less$ on $T_\Sigma X$ (that is, for the induced equivalence relation $\sim$ the underlying poset is $|T_\Vvar X| = |T_\Sigma X|/\sim$ and the operations are induced by those of $T_\Sigma X$).
		This statement follows from Remark~\ref{rem:varieties-sur-reflective}, which implies that a free algebra $T_\Vvar X$ exists, and the unique homomorphism
		\[
		e_X: T_\Sigma X \to T_\Vvar X
		\]
		extending the universal arrow is epic.
		Indeed, the desired preorder is simply
		\[
		s \less t \text{ iff } e_X(s) \leq e_X(t).
		\]
		
		\item The preorder $\Eeq_X$ of the above construction coincides with $\less$ of (2).
		Indeed, if $(s,t) \in \Eeq_X$, then the algebra $T_\Vvar X$ satisfies $s \leq t$ (since it lies in $\Vvar$) and taking the universal map $(\eta_\Vvar)_X: X \to T_\Vvar X$ as the interpretation, we have
		\[
		e_X = (\eta_\Vvar)_X^\sharp
		\]
		(because $e_X$ is a $\Sigma$-homomorphism).
		Since $e_X(s) \leq e_X(t)$, we conclude that $s \less t$.
		
		Conversely, if $s \less t$, which means $e_X(s) \leq e_X(t)$, we verify that every algebra $A \in \Vvar$ satisfies $s \leq t$.
		Let $f: X \to A$ be an interpretation, then the corresponding homomorphism $f^\sharp: T_\Sigma X \to A$ factorises through the reflection of $T_\Sigma X$ in $\Vvar$ in $\Alg(\Sigma)$:
		\[
		\begin{tikzcd}
			T_\Sigma X
			\arrow[r, "e_X"]
			\arrow[rd, "f^\sharp", swap]
			&
			T_\Vvar X
			\arrow[d, "h"]
			\\
			&
			A
		\end{tikzcd}
		\]
		Since $h$ is monotone, the inequality $e_X(s) \leq e_X(t)$ implies $f^\sharp(s) \leq f^\sharp(t)$, as required.	
	\end{enumerate}
\end{proof}

\begin{notation}
\label{not:variety-monad-morphism}
For every variety $\Vvar$ of $\Sigma$-algebras we denote by
\[
c_\Vvar: \Tmon_\Sigma \to \Tmon_\Vvar
\]
the monad morphism whose components are the canonical quotient maps
\[
|T_\Sigma X| \to |T_\Sigma X|/\Eeq_X^o.
\]
\end{notation}

\begin{lemma}
\label{lem:varietal-forgetful-monadic}
For every variety $\Vvar$ the forgetful functor to $\Pos$ is strictly monadic: the comparison functor $K: \Vvar \to \Pos^{\Tmon_\Vvar}$ is an isomorphism.
\end{lemma}

\begin{proof}
For classical varieties see~\cite{maclane:cwm}, Theorem VI.8.1.
The proof for varieties of ordered algebras is completely analogous, one just replaces the equation $\lambda_B = \mu_B$ with the inequation $\lambda_B \leq \mu)B$.
\end{proof}

\begin{remark}[See~\cite{kelly:book}]
\label{rem:spitze-klammern}
\phantom{formatting}
\begin{enumerate}
	\item Recall the \emph{continuation monad} $\spitze{A,A}$ on $\Pos$ associated with every poset $A$: to a poset $X$ it assigns the power of $A$ to the set $\Pos_\ordi (X,A)$ of all monotone maps $f: X \to A$:
	\[
	\spitze{A,A}X = \prod_{\Pos_\ordi (X,A)} A.
	\]
	Denote by $\pi_f: \spitze{A,A} X \to A$ the projection corresponding to $f: X \to A$.
	To every morphism $h: X \to Y$ the monad assigns the morphism $\spitze{A,A}h$ determined by the following commutative triangles:
	\[
	\begin{tikzcd}
		\prod_{\Pos_\ordi (X,A)} A
		\arrow[rr, "\spitze{A,A}h"]
		\arrow[dr, "\pi_{f \comp h}", swap]
		&
		&
		\prod_{\Pos_\ordi (Y,A)} A,
		\arrow[dl, "\pi_f"]
		\\
		&
		A
		&
	\end{tikzcd}
	\qquad f \in \Pos_\ordi(Y,A).
	\]
	The unit is $\langle f \rangle_{f \in \Pos_\ordi (X,A)}: X \to \spitze{A,A} X$, and the multiplication $\mu_X$ is determined by the following commutative triangles:
	\[
	\begin{tikzcd}
		\prod_{\Pos_\ordi (\spitze{A,A}X,A)} A
		\arrow[rr, "\mu_X"]
		\arrow[dr, "\pi_{\pi_f}", swap]
		&
		&
		\prod_{\Pos_\ordi (X,A)} A,
		\arrow[dl, "\pi_f"]
		\\
		&
		A
		&
	\end{tikzcd}
	\qquad f \in \Pos_\ordi(X,A).
	\]
	
	\item It follows from~\cite{dubuc:kan-extensions} that for every monad $\Tmon$ and every poset $A$ there is a bijection between monad morphisms $\Tmon \to \spitze{A,A}$ and algebras of $\Pos^\Tmon$ on $A$.
	This bijection assigns to an algebra $\alpha: TA \to A$ the monad morphism
	\[
	\wh{\alpha}: \Tmon \to \spitze{A,A}
	\]
	with components determined by the following commutative squares:
	\[
	\begin{tikzcd}
		TX
		\arrow[r, "\wh{\alpha}_X"]
		\arrow[d, "Tf", swap]
		&
		\spitze{A,A} X
		\arrow[d, "\pi_f"]
		\\
		TA
		\arrow[r, "\alpha", swap]
		&
		A
	\end{tikzcd}
	\qquad f \in \Pos_\ordi (X,A).
	\]
	Thus if $\Tmon = \Tmon_\Sigma$, then $\wh{\alpha}_X$ assigns to a term $t \in T_\Sigma X$ the tuple $(f^\sharp(t))_{f:X \to A}$.
	
	\item Let $b: \Smon \to \Tmon$ be a monad morphism.
	Every algebra $(A,\alpha)$ in $\Pos^\Tmon$ then yields an algebra $(A, \alpha \comp b_A)$ in $\Pos^\Smon$.
	The following triangle
	\[
	\begin{tikzcd}
		\Smon
		\arrow[rr, "b"]
		\arrow[rd, "\wh{\alpha \comp b_A}", swap]
		&
		&
		\Tmon
		\arrow[dl, "\wh{\alpha}"]
		\\
		&
		\spitze{A,A}
		&
	\end{tikzcd}
	\]
	commutes.
	Indeed, for every poset $X$ and every $f \in \Pos_\ordi (X,A)$ we have
	\[
	\pi_f(\wh{\alpha}_X \comp b_X) = \alpha \comp Tf \comp b_X = \alpha \comp b_A \comp Sf.
	\]
	The same result is obtained by
	\[
	\pi_f (\wh{\alpha \comp b_A}_X) = \alpha \comp b_A \comp Sf.
	\]
	
	\item In particular, let $\Tmon = \Tmon_\Sigma$ for a signature $\Sigma$.
	Given a term $u$ in $T_\Sigma n$, it corresponds to a monad morphism
	\[
	\wt{u}: \Tmon_{\Omega_n} \to \Tmon_\Sigma
	\]
	where $\Omega_n$ is a signature of a single operation $\omega$ of arity $n$.
	Its component $\wt{u}_X: T_{\Omega_n} X \to T_\Sigma X$ assigns to a term $t$ over $X$ (containing the unique operation symbol $\omega$) the $\Sigma$-term obtained by replacing each $\omega$ by the term $u$.
	Thus if a $\Sigma$-algebra $(A,\alpha)$ satisfies an inequation $u_0 \leq u_1$, the inequation $(\wh{\alpha} \comp \wt{u}_0)_X \leq (\wh{\alpha} \comp \wt{u}_1)_X$ holds for all posets $X$.
	Shortly: $\wh{\alpha} \comp \wt{u}_0 \leq \wh{\alpha} \comp \wt{u}_1$.
\end{enumerate}
\end{remark}

\begin{example}
We describe the free-algebra monad of the variety given by a single inequation $u_0 \leq u_1$ in signature $\Sigma$.
Let $u_0,u_1$ be terms with variables $x_0,\dots, x_{n-1}$.
For the signature $\Omega_n$ of a single operation of arity $n$ they can be viewed (via Yoneda lemma) as natural transformations 
\[
u_0,u_1: H_{\Omega_n} \to T_\Sigma.
\]
The corresponding monad morphisms
\[
\wt{u}_0,\wt{u}_1: \Tmon_{\Omega_n} \to \Tmon_\Sigma.
\]
have, in the category of strongly finitary monads, a coinserter we denote as follows:
\[
\begin{tikzcd}
	{\Tmon_{\Omega_n}}
	\arrow[r, "\wt{u}_1", bend left]
	\arrow[r, "\wt{u}_0", swap, bend right]
	&
	{\Tmon_\Sigma}
	\arrow[r, "c"]
	&
	\Tmon
\end{tikzcd}
\]
\end{example}

We verify that this is precisely $c_\Vvar$ above for the variety presented by $u_0 \leq u_1$.

\begin{proposition}
\label{prop:free-algebra-monad}
The above monad $\Tmon$ is the free-algebra monad of the variety presented by the inequation $u_0 \leq u_1$.
\end{proposition}

\begin{proof}
The variety $\Vvar$ presented by $u_0 \leq u_1$ yields a free-algebra monad $\Tmon_\Vvar$.
The proposition will be proved by verifying that $c_\Vvar$ (Notation~\ref{not:variety-monad-morphism}) is a coinserter of $\wt{u}_0$,$\wt{u}_1$ in $\Mndsf(\Pos)$.
From the definition of $c_\Vvar$ we conclude
\[
c_\Vvar \comp \wt{u}_0 \leq c_\Vvar \comp \wt{u}_1.
\]

\begin{enumerate}[(a)]
	
	\item
	
	Given a strongly finitary monad $\Smon = (S,\mu^S,\eta^S)$ and a monad morphism $b: \Tmon_\Sigma \to \Smon$ with
	\[
	b \comp \wt{u}_0 \leq b \comp \wt{u}_1,
	\]
	we prove that $b$ factorises through $c_\Vvar$ via a monad morphism.
	
	For every poset $X$, the free algebra $(SX, \mu^S_X)$ for $\Smon$ yields, since $b$ is a monad morphism, the following algebra for $\Tmon_\Sigma$ on $SX$:
	\[
	\beta_X := T_\Sigma SX \xrightarrow{b_{SX}} SSX \xrightarrow{\mu^S_X} SX
	\]
	From $\alpha_X \comp (\wt{u}_0)_X \leq \alpha_X \comp (\wt{u}_1)_X$ we deduce, using Remark~\ref{rem:spitze-klammern}~(4), that the $\Sigma$-algebra $(SX, \beta_X)$ satisfies the inequality $u_0 \leq u_1$.
	Since the free algebra $(TX, \mu^T_X)$ of $\Vvar$ on $X$ corresponds to the $\Sigma$-algebra
	\[
	T_\Sigma TX \xrightarrow{(c_\Vvar)_{TX}} TTX \xrightarrow{\mu^T_X} TX,
	\]
	we obtain a unique $\Sigma$-homomorphism $\ol{b}_X$ with $\ol{b}_X \comp \eta^T_X = \eta^S_X$:
	\[
	\begin{tikzcd}
		T_\Sigma TX
		\arrow[r, "(c_\Vvar)_{TX}"]
		\arrow[d, "T_\Sigma \ol{b}_X", swap]
		&
		TTX
		\arrow[r, "\mu^T_X"]
		&
		TX
		\arrow[d, "\ol{b}_X"]
		&
		X
		\arrow[l, "\eta^T_X", swap]
		\arrow[ld, "\eta^S_X"]
		\\
		T_\Sigma S X
		\arrow[r, "b_{SX}", swap]
		&
		SSX
		\arrow[r, "\mu^S_X", swap]
		&
		SX
		&
	\end{tikzcd}
	\]
	We verify that these morphisms $\ol{b}_X$ form a monad morphism
	\[
	\ol{b}: \Tmon \to \Smon \text{ with } b = \ol{b} \comp c_\Vvar.
	\]
	\begin{enumerate}[(1)]
		\item The equality $b_X = \ol{b}_X \comp (c_\Vvar)_X: T_\Sigma X \to SX$ holds because both sides are homomorphisms of $\Sigma$-algebras and we have
		\[
		b_X \comp \eta^\Sigma_X = \eta^S_X = \ol{b}_X \comp \eta^T_X = \ol{b}_X \comp (c_\Vvar)_X \comp \eta^\Sigma_X.  
		\]
		
		\item $\ol{b}_X$ is natural in $X$.
		In fact, every morphism $f: X \to Y$ yields a $\Sigma$-homomorphism
		\[
		Tf: (TX, \mu^T_X \comp (c_\Vvar)_{TX}) \to (TY, \mu^T_Y \comp (c_\Vvar)_{TY})
		\]
		Thus, $\ol{b}_Y \comp Tf$ is also a $\Sigma$-homomorphism, and so is $Sf \comp \ol{b}_X: (TX, \mu^T_X \comp (c_\Vvar)_{TX}) \to (SY, \alpha_Y)$.
		Since the domain of both composites is a free algebra of $\Vvar$ on $X$, for proving that they are equal we just need to verify
		\[
		\ol{b}_Y \comp Tf \comp \eta^T_X = Sf \comp \ol{b}_X \comp \eta^T_X.
		\]
		See the following diagram:
		\[
		\begin{tikzcd}
			TX
			\arrow[rr, "\ol{b}_X"]
			\arrow[ddd, "Tf", swap]
			&
			&
			SX
			\arrow[ddd, "Sf"]
			\\
			&
			X
			\arrow[ul, "\eta^T_X", swap]
			\arrow[ur, "\eta^S_X"]
			\arrow[d, "f", swap]
			&
			\\
			&
			Y
			\arrow[dl, "\eta^T_Y", swap]
			\arrow[dr, "\eta^S_Y"]
			&
			\\
			TY
			\arrow[rr, "\ol{b}_Y", swap]
			&
			&
			SY
		\end{tikzcd}
		\]
		
		\item The equality
		\[
		\ol{b} \comp \eta^T = \eta^S
		\]
		follows from the right-hand triangle in the diagram defining $\ol{b}_X$ above.
		
		\item We finally prove
		\[
		\ol{b} \comp \mu^T = \mu^S \comp S\ol{b} \comp \ol{b}T.
		\]
		Consider the following diagram
		\[
		\begin{tikzcd}
			T_\Sigma T X
			\arrow[r, "(c_\Vvar)_{TX}"]
			\arrow[rd, swap, "b_{TX}"]
			\arrow[dd, swap, "T_\Sigma \ol{b}_X"]
			&
			TTX
			\arrow[r, "\mu^T_X"]
			\arrow[d, "\ol{b}_{TX}"]
			&
			TX
			\arrow[dd, "\ol{b}_X"]
			\\
			&
			STX
			\arrow[d, "S \ol{b}_X"]
			&
			\\
			T_\Sigma S X
			\arrow[r, swap, "b_{SX}"]
			&
			SSX
			\arrow[r, swap, "\mu^S_X"]
			&
			SX
		\end{tikzcd}
		\]
		The outward rectangle is the definition of $\ol{b}_X$.
		The left-hand parts commute by (1) and (2).
		Consequently, the desired right-hand square commutes since it does when precomposed by the epimorphism $(c_\Vvar)_{TX}$.
		
	\end{enumerate}

	\item Finally for every monad morphism $b': \Tmon_\Sigma \to \Smon$ factorised as $b' = \ol{b'} \comp c_\Vvar$ we are to verify that
	\[
	b \leq b' \text{ implies } \ol{b} \leq \ol{b'}.
	\]
	This is trivial since the components of $c_\Vvar$ are surjective.
	 
\end{enumerate}

\end{proof}

\begin{construction}
\label{con:monad-coinserter}
The above proposition immediately generalises to sets of inequations.
For every variety $\Vvar$ of $\Sigma$-algebras the free-algebra monad $\Tmon_\Vvar$ is a canonical quotient $c_\Vvar: \Tmon_\Sigma \to \Tmon_\Vvar$ of the free-$\Sigma$-algebra monad, see Notation~\ref{not:variety-monad-morphism}.
We construct monad morphisms $\wt{u}_0, \wt{u}_1: \Tmon_\Omega \to \Tmon_\Sigma$ for some signature $\Omega$ forming a coinserter in $\Mndsf(\Pos)$ as follows:
\[
\begin{tikzcd}
	{\Tmon_{\Omega}}
	\arrow[r, "\wt{u}_1", bend left]
	\arrow[r, "\wt{u}_0", swap, bend right]
	&
	{\Tmon_\Sigma}
	\arrow[r, "c_\Vvar"]
	&
	\Tmon_\Vvar
\end{tikzcd}
\]
Given a collection
\[
u_0^i \leq u_1^i, \qquad i \in I
\]
of inequations specifying the variety $\Vvar$, let $n_i$ be the number of variables on both sides.
We define a signature $\Omega = \{ \gamma_i \}_{i \in I}$, where $\gamma_i$ has arity $n_i$.
By Yoneda lemma we obtain natural transformations $u_0,u_1: H_{\Omega} \to T_\Sigma$, since we have $H_\Omega \cong \coprod_{i\ in I} \Pos(n_i,\blank)$.
Let $\wt{u}_0,\wt{u}_1: \Tmon_{\Omega} \to \Tmon_\Sigma$ be the corresponding monad morphisms.
In the category $\Mndsf(\Pos)$ we form a coinserter
\[
\begin{tikzcd}
	{\Tmon_{\Omega}}
	\arrow[r, "\wt{u}_1", bend left]
	\arrow[r, "\wt{u}_0", swap, bend right]
	&
	{\Tmon_\Sigma}
	\arrow[r, "c"]
	&
	\Tmon
\end{tikzcd}
\]
\end{construction}

\begin{proposition}
For every variety $\Vvar$ of ordered algebras the above monad $\Tmon$ is the corresponding free-algebra monad $\Tmon_\Vvar$.
\end{proposition}

The proof is completely analogous to Proposition~\ref{prop:free-algebra-monad}.

\begin{corollary}
\label{cor:variety-monad-sf}
The free-algebra monad $\Tmon_\Vvar$ of a variety of ordered algebras is strongly finitary.
It follows from the above proposition that we have a coinserter
\[
\begin{tikzcd}
	{T_{\Omega}}
	\arrow[r, "\wt{u}_1", bend left]
	\arrow[r, "\wt{u}_0", swap, bend right]
	&
	{T_\Sigma}
	\arrow[r, "c_\Vvar"]
	&
	\Tmon_\Vvar
\end{tikzcd}
\]
in $[\Pos,\Pos]$.
Hence, $\Tmon_\Vvar$ is strongly finitary by Examples~\ref{ex:strongly-finitary-functors}~(3) and~\ref{ex:free-monad-from-endofunctor}.
\end{corollary}

\begin{example}
	A finitary monad on $\Pos$ need not be strongly finitary.
	(In contrast, every finitary monad on $\Set$ is strongly finitary in the sense of preserving reflexive coequalisers, see~\cite{kurz+rosicky}.)
	
	Denote by $\Vvar$ the category of partial algebras $(A,\alpha)$ where $A$ is a poset and $\alpha$ a monotone function assigning to every pair $a_0 \leq a_1$ in $A$ an element of $A$.
	Morphisms to $(B,\beta)$ are monotone functions $h: A \to B$ such that
	\[
	h \alpha (a_0,a_1) = \beta(h(a_0),h(a_1))
	\]
	holds for all $a_0 \leq a_1$.
	This is a `variety in context' as introduced in~\cite{adamek+chase+milius+schroder}, from which it follows that the forgetful functor  $U: \Vvar \to \Pos$ is finitary monadic, see Theorem 3.24 in op.\ cit.
	The corresponding monad $\Tmon$ assigns to a poset $X$ the poset $TX$ defined by induction as follows:
	\begin{enumerate}
		\item elements of $X$ are terms; they are ordered as in $X$, and
		\item given terms $u_0 \leq u_1$, then $\alpha(u_0,u_1)$ is a term and the ordering is pointwise: for terms $v_0 \leq v_1$ we have $\alpha(u_0,u_1) \leq \alpha(v_0,v_1)$ iff $u_i \leq v_i$ for $i=0,1$.
	\end{enumerate}
	This monad is not strongly finitary because for the 2-chain $P$ given by $x_0 \leq x_1$ it does not preserve its canonical reflexive coinserter (recall~Remark~\ref{rem:poset-canonical-presentation}):
	\[
	\begin{tikzcd}
		\{(x_0,x_0),(x_0,x_1),(x_1,x_1) \}
		\arrow[r, "p_1", bend left]
		\arrow[r, "p_0", swap, bend right]
		&
		\{ x_0, x_1 \}
		\arrow[r, "c"]
		&
		\{ x_0 \leq x_1 \}
	\end{tikzcd}
	\]
	Indeed, every coinserter is surjective, whereas $Tc$ is not: the element $\alpha(x_0,x_1)$ of $TP$ does not lie in the image of $Tc$.
\end{example}

\section{From Strongly Finitary Monads to Varieties}
\label{sec:sf-monads-to-varieties}

We now prove that the results of Section~\ref{sec:varieties-to-sf-monads} can be reversed: for every strongly finitary monad $\Tmon$ a variety is presented with $\Tmon$ as the free-algebra monad.

Recall that given a monad $\Tmon$ every morphism $f: X \to TY$ yields a homomorphism $f^*: (TX,\mu_X) \to (TY,\mu_Y)$ by $f^* = \mu_Y \comp Tf$.
Below we associate with every $n$-ary operation symbol $\sigma$ the term $\sigma(x_i)_{i<n}$ over $V$ (see~Definition~\ref{def:ordered-variety}).

\begin{definition}
\label{def:associated-variety}
For every monad $\Tmon$ on $\Pos$ the \emph{associated variety} $\Vvar_\Tmon$ has the signature $\Sigma$ whose $n$-ary symbols are the elements of $Tn$ ($n \in \Nat$).
The variety is presented by inequations as follows (with $n$ and $m$ ranging over $\Nat$):
\begin{enumerate}
	\item $\sigma(x_i) \leq \tau(x_i)$ for all $\sigma \leq \tau$ in $Tn$;
	\item $k^*(\sigma)(x_i) = \sigma(k_0(x_i),\dots,k_{m-1}(x_i))$ for all $m$-tuples $k: m \to Tn$, $k = (k_0,\dots,k_{m-1})$ and all $\sigma \in Tm$.
\end{enumerate}
\end{definition}

\begin{example}
Every algebra $\alpha: TA \to A$ in $\Pos^\Tmon$ yields a $\Sigma$-algebra in $\Vvar_\Tmon$: given an $n$-ary symbol $\sigma \in Tn$ and an $n$-tuple $f: n \to A$, let $f^+ = \alpha \comp Tf: (Tn, \mu_n) \to (A, \alpha)$ be the corresponding homomorphism for $\Tmon$.
We put
\[
\sigma_A(f) = f^+(\sigma).
\]
To verify that this $\Sigma$-algebra satisfies (1) in Definition~\ref{def:associated-variety}, observe that for every $n$-tuple $f: n \to A$ the corresponding $\Sigma$-homomorphism 
$f^\sharp: T_\Sigma V \to A$ fulfills
\[
f^\sharp(\sigma(x_i)) = f^+(\sigma) \text{ for all } \sigma \in Tn. \tag{3} \label{eq:3}
\]
This equality holds since $\sigma(x_i)$ is the result of the operation $\sigma$ in the algebra $T_\Sigma n$ (Example~\ref{ex:free-monad-from-endofunctor}) on $(x_i)$, thus, $f^\sharp(\sigma(x_i)) = \sigma_A(f(x_i))$.
Given $\sigma \leq \tau$ in $Tn$, then $f^+(\sigma) \leq f^+(\tau)$ since $f^+ = \alpha \comp Tf$ is monotone, thus $f^\sharp(\sigma(x_i)) \leq f^\sharp(\tau(x_i))$ holds.

To verify (2), we need to prove
\[
f^\sharp(k^*(\sigma)(x_i)) = f^\sharp(\sigma( k_0(x_i), \dots, k_{m-1}(x_i) ))
\]
for every $n$-tuple $f: n \to A$. Due to~\eqref{eq:3} above, the left-hand side is
\[
f^+(k^*(\sigma)).
\]
Since $f^\sharp$ is a homomorphism, the right-hand side is
\[
\sigma_A(f^\sharp(k_0(x_i)), \dots, f^\sharp(k_{m-1} (x_i) ) )
\]
which due to \eqref{eq:3} is equal to
\[
\sigma_A(f^+ \comp k) = (f^+ \comp k)^+(\sigma)
\]
Thus we only need to observe that
\[
f^+ \comp k^* = (f^+ \comp k)^+: (Tn, \mu_n) \to (A, \alpha). \tag{4} \label{eq:4}
\]
Indeed, both sides are homomorphisms in $\Pos^\Tmon$, and they are equal when precomposed with the universal map:
\[
f^+ \comp k^* \comp \eta_n = f^+ \comp k = (f^+ \comp k)^+ \comp \eta_n.
\]
\end{example}

\begin{remark}
\label{rem:monad-algebras-varietal}
We can thus consider $\Pos^\Tmon$ as a full subcategory of $\Vvar_\Tmon$.
Indeed, given two algebras $(A,\alpha)$ and $(B,\beta)$ in $\Pos^\Tmon$, then a monotone map $h: A \to B$ is a homomorphism in $\Pos^\Tmon$ iff it is a $\Sigma$-homomorphism:
\begin{enumerate}
	\item Let $h \comp \alpha = \beta \comp Th$.
	Then
	\[
	h \comp f^+ = (h \comp f)^+: (Tn, \mu_n) \to (A,\alpha)
	\]
	because both sides are homomorphisms of $\Pos^\Tmon$ extending $h \comp f$.
	For every $\sigma \in \Sigma_n$ and every $n$-tuple $f: n \to A$ we have
	\begin{align*}
	h(\sigma_A(f)) & = h \comp f^+(\sigma), \text{ by definition of $\sigma_A$} \\
	& = (h \comp f)^+, \text{ as $h \comp f^+ = (h \comp f)^+$} \\
	& = \sigma_B(h \comp f), \text{ by definition of $\sigma_B$.}
	\end{align*}
	Thus $h$ is a $\Sigma$-homomorphism.
	
	\item Let $h$ be a $\Sigma$-homomorphism.
	To prove that $h$ is a homomorphism of $T$-algebras, consider the diagram below for an arbitrary $n \in \Nat$ and $f: n \to A$.
	(Recall that $n$ is the discrete poset on $\{0,\dots,n-1\}$.)
	Since $\Tmon$ is finitary, it is sufficient to show that the desired square commutes when precomposed by $Tf$.
	\[
	\begin{tikzcd}
	Tn \arrow[r, "Tf"] & TA \arrow[r, "\alpha"] \arrow[d, "Th", swap] & A \arrow[d, "h"] \\
	& TB \arrow[r, "\beta", swap] & B
	\end{tikzcd}
	\]
	Indeed, given $\sigma \in Tn$ we have
	\begin{align*}
		\beta \comp Th \comp Tf (\sigma)
		& = (h \comp f)^+(\sigma), \text{ by definition of $(h \comp f)^+$} \\
		& = \sigma_B(h \comp f), \text{ by definition of $\sigma_B$} \\
		& = h(\sigma_A(f)), \text{ since $h$ is a $\Sigma$-homomorphism} \\
		& = h(f^+(\sigma)), \text{ by definition of $\sigma_A$} \\
		& = h(\alpha \comp Tf(\sigma)), \text{ by definition of $f^+$.}
	\end{align*}
\end{enumerate}
\end{remark}

\begin{theorem}
\label{thm:sf-monad-free-algebra-monad}
Every strongly finitary monad on $\Pos$ is the free-algebra monad of the associated variety $\Vvar_\Tmon$.
\end{theorem}

\begin{proof}
\phantom{formatting}
\begin{enumerate}
	\item For every poset $X$ we prove that the free algebra $(TX,\mu_X)$ on $X$ in $\Pos^\Tmon$, considered as a $\Sigma$-algebra, is free on $X$ in $\Vvar_\Tmon$ w.r.t.\ $\eta_X$ as the universal map.
	
	To verify this, we can restrict ourselves to finite posets $X$.
	Then it holds for all posets since $T$ preserves filtered colimits: express $X = \colim_{i \in I} X_i$ as a filtered colimit of finite posets, then $TX = \colim_{i \in I} TX_i$, and from Remark~\ref{rem:monad-algebras-varietal} we conclude that the $\Sigma$-algebra $TX$ is a filtered colimit of $TX_i$ ($i \in I$) in $\Alg(\Sigma)$.
	Thus from $TX_i$ being a free $\Sigma$-algebra on $X_i$ in $\Vvar_\Tmon$ we conclude that $TX$ is a free $\Sigma$-algebra on $X$.
	
	Let $P$ be a finite poset, say, on the set $\{x_0, \dots, x_{n-1}\}$.
	Then its canonical coinserter (Remark~\ref{rem:poset-canonical-presentation}) yields, since $T$ is strongly finitary, the following coinserter
	\[
	\begin{tikzcd}
		Tk
		\arrow[r, "Tp_1", bend left]
		\arrow[r, "Tp_0", swap, bend right]
		&
		Tn
		\arrow[r, "\id"]
		&
		TP
	\end{tikzcd}
	\]
	The free algebras $Tk$ and $Tn$ of $\Pos^\Tmon$ are also free $\Sigma$-algebras in $\Vvar_\Tmon$: see Remark~\ref{rem:monad-algebras-varietal}.
	Given an algebra $A$ of $\Vvar_\Tmon$ and a monotone function $f: P \to A$, we thus have a unique $\Sigma$-homomorphism $f': Tn \to A$ with $f = f' \comp \eta_n$.
	To prove that $f'$ is also a $\Sigma$-homomorphism $f': TP \to A$, it is sufficient to verify
	\[
	f' \comp Tp_0 \leq f' \comp Tp_1: Tk \to A.
	\]
	
	\[
	\begin{tikzcd}
		k
		\arrow[r, "p_1", bend left]
		\arrow[r, "p_0", swap, bend right]
		\arrow[d, "\eta_k", swap]
		&
		n
		\arrow[d, "\eta_n"]
		\arrow[rd, "f"]
		&
		\\
		Tk
		\arrow[r, "Tp_1", bend left]
		\arrow[r, "Tp_0", swap, bend right]
		&
		Tn
		\arrow[r, "f'", swap]
		&
		A
	\end{tikzcd}
	\]
	Thus we need to prove that for each $x \in Tk$ we have $f'(Tp_0(x)) \leq f'(Tp_1(x))$.
	Indeed, this holds for all variables $y_i \in k$:
	\begin{align*}
	f' \comp Tp_0(\eta_k(y_i))
	& = f(p_0(y_i)), \text{ by the diagram above} \\
	& \leq f(p_1(y_i)), \text{ by $f$ being monotone} \\
	& = f' \comp Tp_1(\eta_k(y_i)), \text{ by the diagram above.}
	\end{align*}
	And thus we only need to observe that the set of all $x \in Tk$ with the desired property is closed under the $\Sigma$-operations.
	For every $\sigma \in \Sigma_n$ and every $n$-tuple $(x_i)_{i<n}$ with $f' \comp Tp_0(\sigma_{Tk}(x_i)) \leq f' \comp Tp_1(x_i)$ we have (since $Tp_i$ are homomorphisms of $\Pos^\Tmon$)
	\begin{align*}
	f' \comp Tp_0(\sigma_{Tk}(x_i))
	& = f'(\sigma_{Tn}(Tp_0(x_i))), \text{ by Remark~\ref{rem:monad-algebras-varietal}} \\
	& = \sigma_A(f'(Tp_0(x_i))), \text{ since $f'$ is a $\Sigma$-homomorphism} \\
	& \leq \sigma_A(f'(Tp_1(x_i))), \text{ since $\sigma_A$ is monotone} \\
	& = f' \comp Tp_1(\sigma_{Tk}(x_i)) \text{ as above.}
	\end{align*}
	
	\item The full embedding $E: \Pos^\Tmon \to \Vvar_\Tmon$ of~Remark~\ref{rem:monad-algebras-varietal} is concrete.
	That is, if $U: \Pos^\Tmon \to \Pos$ and $V: \Vvar_\Tmon \to \Pos$ denote the forgetful functors, the triangle
	\[
	\begin{tikzcd}
	{\Pos^\Tmon}
	\arrow[r, "E"]
	\arrow[dr, "U", swap]
	&
	\Vvar_\Tmon
	\arrow[d, "V"]
	\\
	&
	\Pos
	\end{tikzcd}
	\]
	commutes.
	Both $U$ and $V$ are monadic functors by Lemma~\ref{lem:varietal-forgetful-monadic}.
	It follows from (1) that the corresponding monads are isomorphic.
\end{enumerate}
\end{proof}

\begin{notation}
Let $\Var(\Pos)$ denote the category of varieties of ordered algebras and concrete functors.
These are functors $F: \Vvar_1 \to \Vvar_2$ which commute (strictly) with the forgetful functors $U_i: \Vvar_i \to \Pos$:
\[
\begin{tikzcd}
	\Vvar_1
	\arrow[r, "F"]
	\arrow[dr, "U_1", swap]
	&
	\Vvar_2
	\arrow[d, "U_2"]
	\\
	&
	\Pos
\end{tikzcd}
\]
\end{notation}

\begin{theorem}
\label{thm:kv}
The category of varieties is dually equivalent to the category of strongly finitary monads:
\[
\Var(\Pos) \cong \Mndsf(\Pos)^\op.
\]
\end{theorem}

\begin{proof}
\phantom{formatting}
\begin{enumerate}
\item 
Let $F: \Vvar_1 \to \Vvar_2$ be a concrete functor.
The comparison functors $K_i: \Vvar_i \to \Pos^{\Tmon_{\Vvar_i}}$ are isomorphisms of categories by Lemma~\ref{lem:varietal-forgetful-monadic}.
These isomorphisms are concrete: if $U_i': \Pos^{\Tmon_{\Vvar_i}} \to \Pos$ denotes the underlying functor, then $U_i = U_i' \comp K_i$.
From $F$ we thus obtain a concrete functor
\[
\ol{F} = K_2 \comp F \comp K_1^{-1}: \Pos^{\Tmon_{\Vvar_1}} \to \Pos^{\Tmon_{\Vvar_2}}.
\]
\[
\begin{tikzcd}
\Vvar_1
\arrow[rr, "F"]
\arrow[rd, "U_1"]
&
&
\Vvar_2
\arrow[dd, "K_2"]
\arrow[ld, swap, "U_2"]
\\
&
\Pos
&
\\
\Pos^{\Tmon_{\Vvar_1}}
\arrow[rr, swap, "\ol{F}"]
\arrow[ur, "U_1'"]
\arrow[uu, "K_1^{-1}"]
&
&
\Pos^{\Tmon_{\Vvar_2}}
\arrow[ul, swap, "U_2'"]
\end{tikzcd}
\]
The passage $F \mapsto \ol{F}$ is bijective (with the inverse passage $K_2^{-1} \comp (\blank) \comp K_1$) and preserves composition and identity morphisms.

\item 
Given monads $\Tmon_1$, $\Tmon_2$, monad morphisms $\rho: \Tmon_2 \to \Tmon_1$ bijectively correspond to concrete functors from $\Pos^{\Tmon_1}$ to $\Pos^{\Tmon_2}$: the bijection takes $\rho$ to $H_\rho: \Pos^{\Tmon_1} \to \Pos^{\Tmon_2}$ assigning to an algebra $\alpha: T_1 A \to A$ in $\Pos^{\Tmon_1}$ the algebra
\[
T_2 A \xrightarrow{\rho_A} T_1 A \xrightarrow{\alpha} A
\]
in $\Pos^{\Tmon_2}$.
This passage $\rho \mapsto H_\rho$ moreover preserves composition and indentity morphisms.
See~\cite{barr+wells}, Theorem~3.6.3.

\item 
Define a functor
\[
R: \Var(\Pos) \to \Mndsf(\Pos)^\op
\]
on objects by
\[
R(\Vvar) = \Tmon_\Vvar
\]
and on morphisms $F: \Vvar_1 \to \Vvar_2$ by the following rule
\[
R(F) = \rho \text{ iff } H_\rho = \ol{F}.
\]
It follows from (1) and (2) that $R$ is a well-defined full and faithful functor.
Theorem~\ref{thm:sf-monad-free-algebra-monad} tells us that every strongly finitary monad is isomorphic to $R(\Vvar)$ for some variety $\Vvar$.
Therefore, $R$ is an equivalence of categories.
\end{enumerate}
\end{proof}

\section{Lifting Finitary Monads from $\Set$ to $\Pos$}

The examples of varieties of ordered algebras presented so far are all liftings of varieties of classical algebras (over $\Set$).
In the present section we prove that this is no coincidence: there are no other examples.
Since varieties of ordered algebras are in a bijective correspondence with strongly finitary monads on $\Pos$ (and varieties of classical algebras are in a bijective correspondence with finitary monads on $\Set$), an equivalent statement is the following theorem.

\begin{theorem}
\label{thm:strfin-monad-lifting}
Every strongly finitary monad $(T, \mu, \eta)$ on $\Pos$ is a lifting of a finitary monad $(\wt{T}, \wt{\mu}, \wt{\eta})$ on $\Set$: for every poset $X$ the underlying set of $TX$ is $\wt{T} |X|$, and the underlying maps of $\mu_X$ and $\eta_X$ are $\wt{\mu}_X$ and $\wt{\eta}_X$, resp.
\end{theorem}

Before proving this theorem, we explain why we have decided for the above strict variant of lifting.

\begin{remark}
\phantom{formatting}
\begin{enumerate}
	\item There is a less strict concept of a lifting of an ordinary monad $\wt{\Tmon}$ on $\Set$: denote by $U: \Pos \to \Set$ the forgetful functor.
	A monad $\Tmon$ on $\Pos$ is a \emph{non-strict lifting} of $\wt{\Tmon}$ iff there is a natural isomorphism $\phi$
	\[
	\begin{tikzcd}
		\Pos
		\arrow[r, "T"]
		\arrow[d, "U", swap]
		\arrow[rd, phantom, "\phi \lDar"{middle}]
		&
		\Pos
		\arrow[d, "U"]
		\\
		\Set
		\arrow[r, "\wt{T}", swap]
		&
		\Set
	\end{tikzcd}
	\]
	such that the following diagrams commute:
	\[
	\begin{tikzcd}
		&
		U
		\arrow[ld, "U \eta", swap]
		\arrow[rd, "\wt{\eta}U"]
		&
		\\
		UT
		\arrow[rr, "\phi", swap]
		&&
		\wt{T}U
	\end{tikzcd}
	\]
	\[
	\begin{tikzcd}
	UTT
	\arrow[r, "\phi T"]
	\arrow[d, "U \mu", swap]
	&
	\wt{T}UT
	\arrow[r, "\wt{T} \phi"]
	&
	\wt{T}\wt{T}U
	\arrow[d, "\wt{\mu} U"]
	\\
	UT
	\arrow[rr, "\phi", swap]
	&
	&
	\wt{T}U
	\end{tikzcd}
	\]
	\item Given $\phi$ as above, $\Tmon$ is isomorphic to a monad $\Tmon_\ordi$ on $\Pos$ which is a strict lifting of $\wt{\Tmon}$ (i.e., for which the conditions in the above theorem hold).
	Indeed, define $\Tmon_\ordi = (T_\ordi,\mu_\ordi,\eta_\ordi)$ by letting $T_\ordi X$ be the unique poset on the set $\wt{T}|X|$ for which $\phi_X$ carries an isomorphism $TX \cong T_\ordi X$ in $\Pos$.
	Analogously define $T_\ordi$ on morphisms $f: X \to Y$:  the underlying map of $T_\ordi f$ is such that the square
	\[
	\begin{tikzcd}
		TX
		\arrow[r, "Tf"]
		\arrow[d, "\phi_X", swap]
		&
		TY
		\arrow[d, "\phi_Y"]
		\\
		T_\ordi X
		\arrow[r, "T_\ordi f", swap]
		&
		T_\ordi Y
	\end{tikzcd}
	\]
	commutes.
	The unit of $\Tmon_\ordi$ has components $\phi_X \comp \eta_X: X \to T_\ordi X$ and the multiplication $(\mu_\ordi)_X: T_\ordi T_\ordi X \to T_\ordi X$ is the unique monotone map for which the square
	\[
	\begin{tikzcd}
		TTX
		\arrow[r, "\mu_X"]
		\arrow[d, "(\phi * \phi)_X", swap]
		&
		TX
		\arrow[d, "\phi_X"]
		\\
		T_\ordi T_\ordi X
		\arrow[r, "(\mu_\ordi)_X", swap]
		&
		T_\ordi X
	\end{tikzcd}
	\]
	commutes.
	It is easy to see that $\Tmon_\ordi = (T_\ordi,\eta_\ordi,\mu_\ordi)$ is a well-defined monad on $\Pos$ isomorphic to $\Tmon$ via $\phi$, and that it is a strict lifting of $\wt{\Tmon}$.
	
\end{enumerate}
\end{remark}

\begin{proof}[Proof of Theorem~\ref{thm:strfin-monad-lifting}]
In view of Theorem~\ref{thm:kv} it is sufficient to present, for every variety $\Vvar$ of ordered $\Sigma$-algebras, a variety $\wt{\Vvar}$ of non-ordered algebras such that $\Tmon_\Vvar$ is a lifting of $\Tmon_{\wt{\Vvar}}$.
Here $\Tmon_\Vvar$ is the ordinary $\Vvar$-free-algebra monad on $\Pos$, and $\Tmon_{\wt{\Vvar}}$ the $\wt{\Vvar}$-free-algebra (ordinary) monad on $\Set$.
Recall that we consider an arbitrary set as the poset with the trivial order.
\begin{enumerate}
	
	\item For our standard set $V = \{ x_0, x_1, x_2,\dots \}$ of variables in Definition~\ref{def:ordered-variety} we have defined a set $\Eeq_V^o = \Eeq_V \cap \Eeq_V^{-1}$ of equations in Construction~\ref{con:varietal-free-algebra}: they are those equations $s = t$ which every ordered algebra in $\Vvar$ satisfies.
	(Since this is equivalent to satisfying both $s \leq t$ and $t \leq s$.)
	We denote by $\wt{\Vvar}$ the variety of non-ordered algebras presented by $\Eeq_V^o$.
	This clearly implies that every algebra in $\wt{\Vvar}$ satisfies, for every set $X$, all equations $s = t$ for pairs in $\Eeq_X^o$. Moreover, $\Eeq_X^o$ is clearly a congruence on the non-ordered $\Sigma$-algebra $\wt{T}_\Sigma X$ of all $\Sigma$-terms on $X$.
	
	\item Denote by $T_{\wt{\Vvar}} X$ the free algebra of $\wt{\Vvar}$ on the set $X$.
	It can be constructed as the quotient of the non-ordered algebra $\wt{T}_\Sigma X$ modulo the congruence $\Eeq_X^o$:
	\[
	T_{\wt{\Vvar}} X = \wt{T}_\Sigma X / \Eeq_X^o.
	\]
	The proof is completely analogous to that of Theorem~\ref{thm:free-ordered-algebra}.
	We thus conclude that for an arbitrary poset $X$ our choice of $T_\Vvar X$ and $T_{\wt{\Vvar}} |X|$ can be such that the underlying set of $T_\Vvar X$ is $T_{\wt{\Vvar}} |X|$ and all operations are equal.
	The universal arrows $(\eta_\Vvar)_X: X \to |T_\Vvar X|$ and $(\eta_{\wt{\Vvar}})_{|X|}: |X |\to T_{\wt{\Vvar}} |X|$ are both given by forming the equivalence classes of $x \in X$ modulo $\Eeq_X^o$, thus $\eta_{\wt{\Vvar}}$ is the underlying map of $\eta_\Vvar$.
	The multiplication $(\mu_\Vvar)_X: T_\Vvar T_\Vvar X \to T_\Vvar X$ is an interpretation of every term $t \in T_\Vvar X$ over the poset $T_\Vvar X$ of $\Sigma$-terms as a term $\mu_\Vvar(t)$ over $X$ modulo $\Eeq_X^o$.
	This interpretation is independent of the ordering of $X$, shortly, the underlying function of $(\mu_\Vvar)_X$ is the corresponding interpretation $(\wt{\mu}_\Vvar)_{|X|}$ of terms modulo $\Eeq_X^o$ w.r.t.\ $\wt{\Tmon}$.
	
\end{enumerate}
\end{proof}

\begin{definition}
A variety $\Vvar$ of ordered algebras is called a \emph{lifting} of a variety $\wt{\Vvar}$ of classical (non-ordered) algebras if a functor from $\Vvar$ to $\wt{\Vvar}$ is given which is concrete over $\Set$ and takes the free algebra on any poset $X$ to the free algebra on $|X|$.
\end{definition}

\begin{corollary}
Every variety of ordered algebras is a lifting of some classical variety.
\end{corollary}

Indeed, given a variety $\Vvar$, let $\wt{\Tmon}$ be an ordinary monad of $\Set$ such that $\Tmon_\Vvar$ is a lifting of it.
The comparison functor is an isomorphism $K: \Vvar \to
\Pos^{\Tmon_\Vvar}$ concrete over $\Pos$ (Lemma~\ref{lem:varietal-forgetful-monadic}).
And we have a classical variety $\wt{\Vvar}$ with an analogous concrete isomorphism $\wt{K}: \wt{\Vvar} \to \Set^{\wt{\Tmon}}$ over $\Set$.
Define a concrete functor $H: \Pos^\Tmon \to \Set^{\wt{\Tmon}}$ over $\Set$ by the obvious rule: it sends an algebra $\alpha: TA \to A$ to $\alpha: \wt{T}|A| \to |A|$.
The desired functor is $\wt{K}^{-1} \comp H \comp K: \Vvar \to \wt{\Vvar}$.

\begin{example}
\label{ex:pos-monad-not-strongly-finitary}
We present a finitary lifting of a monad on $\Set$ to $\Pos$ which is not strongly finitary.

Consider all ordered algebras on two binary operations $+$ and $*$.
The full subcategory on all algebras for which the implication
\[
x \leq y \implies x+y \leq x*y
\]
holds yields the following monad $\Tmon$ on $\Pos$.
Given a poset $X$, the poset $TX$ contains all terms with variables in $X$ using $+$ and $*$, where the order on $TX$ is the smallest one such that
\begin{enumerate}
	\item $x \leq y$ in $X$ implies $x \leq y$ in $TX$,
	\item $+$ and $*$ are monotone, and
	\item $t + s \leq t*s$ for all terms $t \leq s$ in $TX$.
\end{enumerate}
Thus $\Tmon$ is a lifting of the monad on $\Set$ corresponding to two binary operations (and no equations).

The monad $\Tmon$ is not strongly finitary.
For example, it does not preserve the canonical coinserter (recall Remark~\ref{rem:poset-canonical-presentation}) of the chain $\Two$ given by $0 < 1$:
\[
	\begin{tikzcd}
	3
	\arrow[r, "p_1", bend left]
	\arrow[r, "p_0", swap, bend right]
	&
	2
	\arrow[r, "\id"]
	&
	\Two
\end{tikzcd}
\]
Indeed, in $T\Two$ we have $0 + 1 < 0 * 1$.
In contrast, this does not hold in the coinserter of $Tp_0$ and $Tp_1$.
We can describe the order of that coinserter as the smallest one that, besides conditions (1)-(3) above, also fulfills $t\leq s$ for terms such that $s$ is obtained by changing some $0$ in $t$ to $1$.
The down-set of the term $0*1$ in that coinserter consists of the following terms
\[
0+0 < 0*0 < 0*1.
\]
Thus, $\Tmon$ does not preserve the coinserter of $p_0$ and $p_1$.
\end{example}

\section{Conclusions}

Kelly and Power proved that every finitary monad $\Tmon$ on $\Pos$ has a presentation as a coequaliser of a parallel pair of monad morphisms between free monads on generalised signatures, see~\cite{kelly+power:adjunctions}.
In the present paper we derive an analogous result for strongly finitary monads: each such monad has a presentation as a coinserter of a parallel pair of monad morphisms between free monads $\Tmon_\Sigma$ on (classical) signatures $\Sigma$, see Construction~\ref{con:monad-coinserter}.
The move from coequalisers to coinserters is needed since the signatures used in~\cite{kelly+power:adjunctions} were substantially more general than those we use here: they were collections $\Sigma = (\Sigma_\Gamma)_{\Gamma \in \Pos_\fp}$ of posets $\Sigma_\Gamma$ indexed by finite posets.
However, the proof method we use is closely related to that in~\cite{kelly+power:adjunctions}.

We have proved that for (classical) varieties of ordered $\Sigma$-algebras the corresponding free-algebra monad on $\Pos$ is strongly finitary, ie.\ finitary and preserving reflexive coinserters.
Using this we proved that the category of
varieties of ordered algebras is dually equivalent to the category of
strongly finitary monads on $\Pos$.

In the future we plan extending our results to strongly finitary monads on more general $\V$-categories for closed monoidal categories $\V$, e.g.\ the category of small categories.
For general $\V$ it is interesting to know under which conditions strongly finitary functors are precisely the finitary ones preserving reflexive coinserters.
But the main question is whether strongly finitary monads correspond again to `naturally` defined varieties of algebras in $\V$.

\end{document}